\documentclass[a4paper,12pt]{amsart} 
\usepackage{a4wide}

\usepackage{graphicx} 

\usepackage[utf8]{inputenc}

\usepackage{amsmath}
\usepackage{amssymb}

\usepackage{hyperref}
\hypersetup{
	colorlinks,
	citecolor=black,
	filecolor=black,
	linkcolor=black,
	urlcolor=black
}

\usepackage{url}

\usepackage{tikz}

\tikzstyle{level 1}=[level distance=3.5cm, sibling distance=3.5cm]
\tikzstyle{level 2}=[level distance=3.5cm, sibling distance=2cm]

\tikzstyle{node} = [text width=7em, text centered]
\tikzstyle{leaf} = [circle, minimum width=3pt,fill, inner sep=0pt]

\def\cb{\mathcal{B}}

\def\cd{\mathcal{D}}

\def\cf{\mathcal{F}}

\def\ch{\mathcal{H}}

\def\ct{\mathcal{T}}
\def\cu{\mathcal{U}}
\def\cv{\mathcal{V}}

\def\w{{\omega}}

\def\bbn{\mathbb{N}}

\def\bbr{\mathbb{R}}
\def\bbz{\mathbb{Z}}

\def\clo #1{{\overline{#1}}}

\def\se{\subseteq}

\def\with{{}^\smallfrown}

\newtheorem{definition}{Definition}[subsection]
\newtheorem{theorem}{Theorem}[subsection]
\newtheorem{lemma}{Lemma}[subsection]
\newtheorem{claim}{Claim}[subsection]

\newtheorem{proposition}{Proposition}[subsection]
\newtheorem{cor}{Corollary}[subsection]

\newtheorem{remark}{Remark}[subsection]
\newtheorem{example}{Example}[subsection]

\def\then{\longrightarrow}
\def\iff{\longleftrightarrow}

\def\<{\langle}
\def\>{\rangle}

\title{Journey into special $T_1$-spaces}

\author{Robert Rałowski}
\address{Robert Rałowski, Department of Pure Mathematics, Faculty of Mathematics, Wrocław University of Science and Technology, 50-370 Wrocław, Poland}
\email{robert.ralowski@pwr.edu.pl}

\date{}

\begin{document}

\begin{abstract}
In this survey, we review certain types of special \(T_{1}\) spaces and their associated fixed-point theorems. Kupka introduced the notion of a feeble topological contraction, which naturally generalizes Lipschitz contractions defined on metric spaces. Specifically, Kupka established a fixed-point theorem for feeble topological contractions possessing a closed graph within the product of arbitrary \(T_{0}\) spaces. Furthermore, the \(T_{1}\) separation axiom is shown to guaranty the uniqueness of such fixed points. Finally, we discuss peripheral Hausdorff and locally Hausdorff spaces within the context of these fixed-point results.
\end{abstract}

\maketitle

\begin{center}
    \thispagestyle{empty}
    \vspace*{\fill}
    To my friend Alexander Cale Pavlovi\'{c} (1974-2025)
    \vspace*{\fill}
\end{center}
This paper is dedicated to Alexander Cale Pavlovi\'{c}. A wonderful man has left us. Alexander Pavlovi\'{c} will remain in our memories as a warm, charming, and humorous person. Alexander leaves behind an outstanding scientific legacy, he was a truly brilliant mathematician.

\tableofcontents

\section{Introduction}
Topology is included in the broad spectrum of Alexander Cale Pavlović's scientific research.
Many of Alexander's papers concern topological spaces that are not necessarily Hausdorff spaces (see, for example, \cite{kurilic2014convergence, pavlovic2016local} and \cite{kurilic2024closed}). In this survey, I investigate topological spaces satisfying a separation axiom weaker than the Hausdorff property.

In 1920, Stefan Banach, in his doctoral thesis, proved the celebrated fixed point theorem.

\begin{theorem}[Banach (1920)] If $(X,d)$ is a complete metric space and $f:X\to X$ is a Lipschitz contraction, i.e.
$$
(\exists c\in[0,1))(\forall x,y\in X)\; (d(f(x),f(y)) \le c \cdot d(x,y)).
$$
then there exists a unique $x\in X$ s.t. $x = f(x).$
\end{theorem}

In literature, we can find many generalizations of this theorem. Much of it concerns Hausdorff spaces. This article is devoted to some class $T_1$ spaces in which we can provide results similar to the Banach fixed point theorem. Of course, if we are outside of metric space, we must redefine the notion of contraction. Ivan Kupka founded an analogy of the mentioned notion, called feebly topological contractive see \cite{K}.

\begin{definition}[Feebly contraction] Let $X$ be a topological space and $f:X\to X.$ W. We say that $f$ is a feebly topological contractive if
$$
(\forall \cu )(\forall x,y\in X) \; (\cu \text{ open cover of }X \then (\exists U\in\cu)(\exists n\in\w)\ f^n [\{x,y\}] \se U).
$$
\end{definition}
Here $f^n = \underbrace{f\circ\ldots\circ f}_{n}$ stands for $n$-th iteration of $f.$

Kupka's result concerns topological spaces $T_1.$ As an example, we will give a version related to functions defined on such spaces.

\begin{theorem}[Kupka (1998)]\label{Kupka-function-thm} If $X$ is $T_0$ topological space, $f:X\to X$ is a feebly topological contraction with closed graph then There exists $x\in X$ s.t. $x = f(x).$ Moreover, if $X$ is $T_1$ space then such $x\in X$ is unique.
\end{theorem}

In the separate section \ref{Kupka-section}, we formulate and prove the strongest version of the Kupka fixed-point theorem.

Let us observe that every continuous map between Hausdorff spaces has a closed graph. Then we have an immediate corollary.
\begin{cor} Every continuous feebly topological contraction on a Hausdorff space has a unique fixed point.
\end{cor}

The assumption that $f \se X\times X$ is closed is quite strong condition in $T_1$ spaces. Even the identity on $\w$ with the co-finite topology (every nonempty open set is co-finite in $\w$) does not have a closed graph in $\w\times\w$ with the product topology but is still continuous. 

The following example shows that the assumption in the Kupka Theorem is necessary.
\begin{example} Let $(\w,\tau)$, where  $\tau = \{\emptyset\}\cup\{U\in P(\w):U^c \text{ is finite}\}.$ Let $f:X\to X$ s.t. for every $n\in\w $ $f(n) = n+1.$ It is easy to check that $f$ is a continuous feeble contraction, but $f\se \w\times\w $ is not closed.
\end{example}

Another concept of topological contraction is realized on compact topological spaces defined in \cite{MR2}. In previously defined spaces, the continuous contraction plays a role in fixed-point theorems. Here, closed topological contractions are used in the theorem for which we have a fixed point.

\begin{definition} If $X$ is a $T_1$ compact space the $f:X\to x$ is a topological contraction if for every open cover $\cu$ of $X$ there are $n\in\w$ and $U\in \cu$ s.t. $f^n[X] \se U.$
\end{definition}

Here we have
\begin{theorem} For every closed topological contraction on the compact $T_1$  there is an exactly one fixed point of $f.$
\end{theorem}

One can weak the notion of topological contraction by changing  any open cover to the two element open cover $\cu = \{\{x\}^c, \{y\}^c\}$ for some fixed distinct $x,y$ in our space. The above theorem remains true in such open cover case. 

Our aim in this paper is isolate a class of $T_1$ topological spaces such that some kind contraction fixed point theorem is valid.

\section{Kupka fixed-point theorem}\label{Kupka-section}
In the Introduction, we mentioned the Kupka Fixed-Point Theorem  \ref{Kupka-function-thm} related to feeble contraction. The strongest version 
is connected with feeble contractive multifunctions. Now we recall this notion in topological spaces.

The proof of this Theorem is short and elegant, as we show.

In full generality, we can replace $f:X\to X$ with $F:X\to P(X)\setminus\{\emptyset\}.$ Function $F$ that we will call multifunction on $X.$ A feeble topological contractive map $f:X\to X$ has to be changed. For fixed $F:X\to P(X)\setminus\{\emptyset\}$, we can define $\hat{F}:P(X)\to P(X)$ such that for every nonempty $A\in P(X)$ set,
$$
\hat{F}(A) = \bigcup F[A] = \bigcup\{ F(x): x\in A\}.
$$
Let us observe that for every nonempty $A,B\in P(X)$, if $A\se B$ then $\hat{F}(A) \se \hat{F}(B).$ Moreover, if $x\in X$, then $\hat{F}(\{x\}) = F(x).$ 

\begin{definition}[Feeble contractive multifunction]\cite[Definition 3]{K} Let $X$ be a fixed topological space. Then  $F:X\to P(X)\setminus\{\emptyset\}$ is feeble topological contractive multifunction on $X$ if for every open cover $\cu$ of $X$, we have
$$
(\forall x, y\in X)(\exists U\in\cu)(\exists n\in\w)\; (\hat{F}^n(\{x\})\se U\land \hat{F}^n(\{y\})\cap U\ne\emptyset).
$$                                                      
\end{definition}

A multifunction $F:X\to P(X)$ has a closed graph if the set
$$
\{(x,y)\in X\times X: \; y\in F(x)\}
$$
is closed in $X\times X.$

\begin{theorem}[Kupka]\cite[Theorem 2]{K} If $X$ is a topological space $T_0$, and $F:X\to P(X)\setminus\{\emptyset\}$ is a feeble topological contractive multifunction with a closed graph on $X$, then there exists $x\in X$ such that $x\in F(x)$. Moreover, if $X$ is $T_1$, then such a point is unique and $F(x) = \{x\}$. 
\end{theorem}

\begin{proof}                                                                       
Let us assume that $F$ has no fixed point, i.e., 
$$
(\forall x\in X)\ \{x\}\times F(x) \cap \{ (t,t):\ t\in X \} = \emptyset.
$$
But by assumption, the set
$$
G_F = \{(x,y)\in X\times X: \; y\in F(x)\}
$$
is closed in $X\times X.$ Then we can find a family of open nonempty sets
$$
\cu = \{ U_x:\; x\in X\}
$$
such that $x\in U_x$ for every $x\in X$ and $\bigcup \{ U_x\times U_x: x\in X\}\subseteq G_F^c.$ Now let us choose any $x\in X$ and $y\in F(x)$; then there are $U\in\cu$ and some $n\in \w$ such that 
$$
\hat{F}^n(\{x\})\se U\land \hat{F}^n(\{y\})\cap U\ne\emptyset
$$
Let us observe that $ \hat{F}^n(\{y\})\se \hat{F}^n\circ\hat{F}(\{x\})= \hat{F}\circ \hat{F}^n(\{x\})\se F[U].$ Then $U\cap F[U]\ne\emptyset$, which gives us that
$$
\emptyset\ne G_F\cap (U\times U)\se G_F\cap \bigcup\{U_x\times U_x:\; x\in X\},
$$
leads to a contradiction.

Now, let us go into the last part. Let us assume that our space is $T_1$ and there are distinct $x,y\in X$ such that $\{x,y\}\se F(x).$ Then we can find an open cover $\cu = \{ X\setminus \{x\}, X\setminus \{y\}\}$ of $X.$ But $F$ is a feeble contraction; then for some $n\in \w$, $\hat{F}^n(\{x\})\se X\setminus\{x\}$ or $\hat{F}^n(\{x\})\se X\setminus \{y\}.$ But $\{x,y\} \se \hat{F}^k(\{x\})$ for every positive integer $k,$ leads to a contradiction.

Using a similar argument, we can show that the $x\in X$ such that $F(x) = \{x\}$ is unique.
\end{proof}

\section{Peripherally Hausdorff spaces}
We start this section with the concept of a peripherally Hausdorff space. Peripherally Hausdorff spaces is a proper subclass of the class of $T_1$-spaces. In this sort of space, the fixed-point theorem is fulfilled for continuous maps instead of having the closed graph property. In general cases, those spaces are not Hausdorff. But the rank of ordinal numbers tells us how far we are from being Hausdorff.

We adopt some notations. Namely, for any topological space $(X,\tau)$ and $y\in X$, we denote by $\tau_y=\{U\in\tau:\; y\in U\}$ the set of all neighborhoods of $y$. Sometimes we write $\tau(X),$ instead of $\tau$ to emphasize which space is being considered. For any $A\se X$ by $\clo{A}$ (or  $\clo{A}^{_X}$), we denote the closure of $A$ in space $(X,\tau).$

\begin{definition} Let $(X,\tau)$ be a topological space and $x\in X$, then the set
$$
[x]_{(X,\tau)} = \bigcap \{\clo{U}:\; x \in U \land U \in \tau\} 
$$
we call a class of $x.$ Sometimes we will use shorter notation for the class of $x$ as $[x]$ or $[x]_X.$
\end{definition}

Now let us observe that for every Hausdorff space, we have $[x]_{(X,\tau)} = \{x\}$ for every $x$ of our space.

It is easy to see that the relation $(X,\preceq)$ is reflexive and symmetric but not necessarily transitive, where
$$
(\forall x,y\in X)\; (x\preceq y \iff x \in [y]_{(X,\tau)}).
$$

Now let $(X,\tau)$ be any topological space and $Y \se X.$ Then we define a relative topology on $Y$
$$
\tau\restriction Y = \{U\cap Y: U \in \tau \}.
$$
Then for any $x\in X$ we can consider $[x]_{(X,\tau)}$ as a topological space with relative topology $\tau([x]_{(X,\tau)}) = \tau\restriction ([x]_{(X,\tau)}).$

Now we are ready to introduce peripherally Hausdorff space.

\begin{definition}[Peripherally Hausdorff space] Let $(X,\tau)$ be an $T_1$-topological space. Then for every ordinal number $\alpha \in On$ we define:
\begin{description}
    \item[$\alpha = 0$] $(X,\tau)\in H_0$ iff $X=\{x\}$ ant $\tau=\{\emptyset,X\},$ or
    \item[$0<\alpha$] if $H_\beta$ is defined for each $\beta< \alpha$ then 
    $$
    (X,\tau) \in H_\alpha \iff (\forall x\in X)( \exists \beta < \alpha)\;\;([x]_{X,\tau}\subsetneq X) \land ( ([x]_{X, \tau}, \tau([x]_{X,\tau}) ) \in H_\beta).
    $$
\end{description}
If $(X,\tau) \in H_\alpha$ then we can read this as $\alpha$-Hausdorff space.

We say that topological $T_1$ space $(X,\tau)$ is peripherally space if exists ordinal number $\alpha$ s.t. $(X,\tau) \in H_\alpha$ or shortly if  $(X,\tau)$ is $\alpha$-Hausdorff space for some ordinal $\alpha.$
\end{definition}
Let us observe that for each ordinals $\beta, \alpha,$ if $0<\beta\le \alpha$ then $H_\beta \se H_\alpha.$ If $X$ is a Hausdorff space but not singleton then $X \in H_1$ and if $X\in H_1$ then $X$ is a Hausdorff space.

For every peripherally Hausdorff space $X$ we define $rank(X)$ as least ordinal $\alpha $ s.t. $X\in H_\alpha.$
Singletons has rank $0$ but every peripherally Hausdorff space which contains at least two elements has rank equal to $1.$

Now we present an example space with rank equal to $2.$

\begin{example}[Niemytzki like half-plane]\label{Niemyrzki_half_plane}  Set $X = \bbr\times [0,\infty).$ Topological subbase of $X$ is defined as follows, for points with positive second axis we equipe euclidean neighbourhoods. For points of form $(x,0)$ where $x\in\bbr$ and $\alpha\in (0,\pi/2)$ we define 
$$
U_\alpha(x,0) = \{(r,s)\in X:\; (r,s) = (x,0) \lor (s\ne |x| \land \tan\alpha \cdot |r|< s)\}.
$$
It is easy to see that 
$$
[(x,y)] = \begin{cases}
    \{(x,y)\}\cup \{(x,0)\} & \text{ for } y>0\\
    \bbr\times\{0\} \cup \{x\}\times (0,\infty) & \text{ for } y = 0
\end{cases},
$$
which is $T_2$ space.

Then $X$ is $2$-Hausdorff, locally Hausdorff space but not $T_2.$
\end{example}

\begin{remark} In Niemytzki-like example we we change the open basic set at point $(x,0)$ as follows for any $\alpha\in (0,\pi/2)$ and $r>0$ set
$$
U_\alpha(x,0) = (x-r,x+r)\times \{0\}\cup \{(r,s)\in X:\; (r,s) = (x,0) \lor (y\ne |x| \land \tan\alpha \cdot |r|< y)\}.
$$
Then resulting space is $T_1$ and $2$-Hausdorff space but not locally Hausdorff space.
\end{remark}

\begin{remark} If we consider $\omega$  with topology s.t. every nonempty open set is cofinite then for every $n\in\w$ we have $[n]_\w = \w.$ Thus this space is $T_1$ but not peripheral Hausdorff space.
\end{remark}

The next theorem tells us that the range of rank function is a whole class of all ordinals.

\begin{theorem}\label{H_alpha}
For every $\alpha\in\textnormal{On}$ there exists a peripherally Hausdorff  space $X$ such that
$$\textnormal{rank}(X)=\alpha.$$ 
\end{theorem}

\begin{proof} The proof goes by induction with respect to $\alpha\in\textrm{On}$. For $\alpha=0,1$  the conclusion is obvious.

Now let us assume that the conclusion of the theorem holds for some $\alpha\geq1$. Let $\textnormal{rank}(X)=\alpha$ for some topological space $X$. Let $\{y_0, y_1,y_2\ldots\}\cap X=\emptyset$ and $y_i\neq y_j$ for $i\neq j$. Let $Y=X\cup\{y_i:i\in\mathbb{N}_0\}.$ Let us define a topology on $Y$ in the following way: let a local base at any point $x\in X$ consist of the sets of the form $U\cup A$ where $A$ is any co-finite subset of $\{y_1,y_2,\ldots\}$ and $U$ is an open neighbourhood of $x$ in $X$. Let a local neighbourhood base of $y_0$ consist of sets $A\cup\{y_0\}$ where again  $A$ is any co-finite subset of $\{y_1,y_2,\ldots\}$.
Let singletons $\{y_1\}, \{y_2\},\ldots$ be open sets.  For $x\in X$:
$$[x]_Y=X\cup\{y_0\}$$
and 
$$[y_0]_Y=X\cup\{y_0\},$$
$$[y_i]_Y=\{y_i\},$$
$i=1,2,\ldots.$
As in the space $X\cup\{y_0\}$ the singleton $\{y_0\}$ is clopen the Hausdorff rank of $[x]_Y$ is equal to the Hausdorff rank of $X$, namely to $\alpha$. Hence 
$$\textrm{rank}(Y)=\alpha+1.$$

Let now $\gamma$ be a limit ordinal and we assume that for all $\alpha<\gamma$ there exists a space $X_\alpha$ whose Hausdorff rank is equal to $\alpha$. Then let $Y$ be a topological sum of the spaces $X_\alpha$, $\alpha<\gamma$. Then for $y\in X_\alpha$ we have
$$[y]_Y=[y]_{X_\alpha}.$$ Hence 
$$\textrm{rank}(Y)=\gamma.$$   

\end{proof}

The notion of peripherally Hausdorff, more exactly that of $\alpha$-Hausdorff, is hereditary. Namely the following theorem holds.
\begin{lemma} If $(X,\tau(X))$ is a peripherally Hausdorff space and $Y\se X$ is nonempty then $(Y,\tau(Y))$ is a peripherally Hausdorff space also, where $\tau(Y) = \tau(X\restriction Y) = \{ U\cap Y: \ U\in \tau(X) \}.$ Moreover, if $(X,\tau(X)) \in H_\alpha$ for some ordinal number then $(Y,\tau(Y)) \in H_\alpha$ also.
\end{lemma}
\begin{proof} Induction over $On.$ If $X\in H_0$ and $\emptyset\ne Y\se X$ then $Y = X =\{x\}$ for some $x.$ Let $\alpha >0$ and for every $\beta < \alpha$ Lemma is true.
		Let $Y\se X$ and $X\in H_\alpha.$ If we show that for every $x\in Y$ $[x]_Y \in \mathcal{F}_\beta$ for some $\beta< \alpha$ then $Y\in H_\alpha$ and so $Y$ is a peripherally Hausdorff space.
		
		 Let us choose any $x\in Y.$ Let us recall that
		 $$
		 [x]_X = \bigcap\{ \clo{U}^{_X}:\; x \in U\in \tau(X)\},
		 $$
		 and
		 $$
		 [x]_Y = \bigcap\{ \clo{(U\cap Y)}^{_Y}:\; x \in U\in \tau(X)\}.
		 $$
		 \begin{claim}\label{funy clases inclusion}
		 	$[x]_Y\se [x]_X.$
		 \end{claim}
		 \begin{proof}[Proof of Claim] Let $t\in [x]_Y$, $U,V\in\tau(X)$ s.t. $x\in U,$ and $t\in V$. Then $t \in V\cap Y \in \tau(X\restriction Y).$ Because $t\in \clo{U\cap Y}^{_Y}$ then $(V\cap Y) \cap (U\cap Y) \ne \emptyset.$ Then $V\cap U\ne \emptyset$ for arbitrary $V\in \tau(X)$ s.t. $t \in V.$ Then $t \in \clo{U}^{_X}$ for any $U\in \tau(X)$ with $x\in U.$ Finally, we have that $t\in [x]_X.$
\end{proof}
	 
\begin{claim}\label{funy clases topology}
	 	$$
	 	\tau([x]_X)\restriction [x]_Y = \tau([x]_Y),
	 	$$
	 	where 
	 	$$
	 	\tau([x]_X) = \{ U\cap [x]_X: U\in \tau(X)\}, \ \tau([x]_Y) = \{ U\cap [x]_Y: U\in \tau(Y)\},\; \tau(Y) = \tau(X)\restriction Y.
	 	$$
	 \end{claim}
	 \begin{proof}[Proof of Claim] For any $V$ we have
	 	\begin{align*}
	 		V \in \tau([x]_X)\restriction [x]_Y \iff & (\exists U_0\in \tau([x]_X))\; (V = U_0\cap [x]_Y) \\
	 		\iff & (\exists U_1\in \tau(X))\; (U_0 = U_1 \cap [x]_X\land V = U_0\cap [x]_Y) \\
	 		\iff & (\exists U_1\in \tau(X))\; (V =( U_1 \cap [x]_X)\cap [x]_Y )\\
	 		\iff & (\exists U_1\in \tau(X))\; (V = U_1 \cap [x]_Y)\\
	 		\iff & (\exists U_1\in \tau(X))\; (V = (U_1 \cap Y) \cap [x]_Y \iff V \in \tau([x]_Y)).
	 	\end{align*}
	 \end{proof}
		But $[x]_X \in H_\beta$ for some $\beta<\alpha$ by assumption that $X\in H_\alpha.$ By the Claim \ref{funy clases inclusion} and Claim \ref{funy clases topology}  and induction assumption we have that $[x]_Y\in H_\beta$ for some $\beta<\alpha.$ We have shown that for every $x\in Y$ there exists $\beta < \alpha$ such that $[x]_Y\in H_\beta$. Then we conclude that $Y\in H_\alpha$ also.
\end{proof}

\subsection{Fixed point theorem}

In Kupka Theorem a feebly topological contraction was used. Moreover, in his theorem an important assumption is that our feebly contraction has to closed graph. In Hausdorff spaces  continuity implies this condition. This is not  true in the case of topological space is $T_1$ in general. Thus we want to prove the fixed point theorem for continuous maps defined on peripherally Hausdorff spaces. To do we need a little stronger condition for being contraction.

\begin{definition}[feebly+ contraction] Let $(X,\tau)$ be a topological space, and $f:X\to X.$ We say that $f$ is a feebly+ topological contraction iff 
$$
(\forall \cu)(\forall x,y\in X)\; (\cu \text{ open cover of } X )\then (\exists U\in \cu) (\forall^\infty n\in\w)\; (f^n[\{x,y\}] \in U).
$$
Here $f^n = \underbrace{f\circ\ldots\circ f}_n.$
\end{definition}
Every feebly+ topological contraction is feebly topological contraction, of course, but we do not know whether the two notions are different. Observe that for any Banach space every Lipschitz contraction with a constant $c\in [0,1)$ is a weakly+ topological contraction.

Now we give an example which shows that continuity of function does not implies closedness of its graphs even if our function is a feebly+ topological contraction.
\begin{example}\label{no_closed_graph} Let $X = \{-1\}\cup [0,1].$ Let the topology of the interval $[0,1]$ be euclidean. For positive real $r>0$ set
$$
U(-1,r) = (\{-1\}\cup (0,r))\cap X
$$
as a basic open neighborhood of the $-1.$ Here for every positive $x\in X$ we have $[x] = \{x\}$ but $[-1] = \{-1,0\} = [0]$ with discreete topology. Then $X$ is peripherally Hausdorff but not Hausdorff space. Now, let us define $f:X\to X$ as follows:
$$f(x) = 
\begin{cases}
    0   & x = -1,\\
    \frac{1}{2}\cdot x  & x\in [0,1]_.
\end{cases}
$$
Let observe that point $(-1,-1)\notin f$ but for every open set containg $(-1,-1)$ contains set $\{(-1,-1)\}\cup (0,r)\times(0,r)$ for some positive $r,$ which intersect $f.$ Then $f$ is not a closed subset of $X^2.$ It is easy to show that $f$ is continuous feebly+ topological contraction. Moreover, we have $0 = f(0).$
\end{example}

Now we formulate the following fixed point theorem for peripherally Hausdorff spaces. 

\begin{theorem}\cite[Theorem 7.7]{MR1}\label{funy fixed point} For every peripherally Hausdorff space $X,$ every continuous feeble$+$ topological contraction on $X$ has a unique fixed-point.
\end{theorem}
\begin{proof} This proof goes by transfinite induction over the class of all ordinal numbers. If $\alpha = 0$ then $\alpha$-Hausdorff space is a singleton and theorem follows.

Now let $\alpha>0$ be a nonzero ordinal number, and for every $\beta <\alpha$ theorem, it is true for each $\beta$-Hausdorff space.

We have the following Claim.
\begin{claim}\label{claim-1-alfa} There is $y\in X$ such that for every open $U$ there exists $x\in U$ such that $\{f^n(x):n\in\w\}\se U.$
\end{claim}
    \begin{proof}[Proof of Claim] If not, then for every $y\in X$ there exists an open $U_y$ such that $y\in U_y$ and
			$$
			(\forall x\in U_y)(\exists n\in\w)\ (f^n(x)\notin U).
			$$
			But 
			$$
			\{ U_y: y\in X\}
			$$
			forms an open cover of $X.$ Set $u\in X.$ Because $f$ is a contraction, then there exists $n_0$ and $y\in X$ such that 
			$$
			(\forall n>n_0)\ \{f^n(u),f^{n+1}(u)\} \se U_y.
			$$
			Set $x = f^{n_0+1}(u)$ then $\{f^n(x):n\in\w\} \se U_y,$ contradiction.
    \end{proof}
		By the above Claim \ref{claim-1-alfa} find $y\in X$ s.t.
		$$
		(\forall U\in\tau)\; (y\in U\then (\exists x\in X)\; (\{f^n(x):n\in\w\} \se U)).
		$$
		
\begin{claim}\label{claim-2-alfa} We have that for every $n\in \w$ $f^n(y) \in [y]$.
\end{claim}
    \begin{proof}[Proof of Claim] Let observe that $f^0(y) = y  \in [y].$ Assume that $n\in\w$ is a smallest s.t. $f^n(y)\notin [y].$ Let $z = f^n(y)$ then for some open $U$ s.t. $y\in U$ $z\notin \clo{U}.$ Then by the continuity of map $f^n$ there is an open $V\se U$ containing $y$ s.t. $z\in f^n[V]\se \clo{U}^c.$ But there is $x\in X$ s.t.
			$$
			\{ f^m(x): m\in\w\} \se V.
			$$
			Then $f^n(x) \in V\se U$ and $f^n(x) \in f^n[V] \se \clo{U}^c$ what is impossible.
    \end{proof}
		
\begin{claim}\label{claim-3-alfa} $\exists ! x\; x = f(x).$
\end{claim}
    \begin{proof}[Proof of Claim] Let us consider set $W = \clo{\{f^n(y):\;n\in\w\}}.$ Then $W\se [y]$ and by the continuity of $f,$ $f[W] \se W.$ Moreover, $f\restriction W$ is a feeble$+$ topological contraction on $W.$ To see this let $a,b\in W$ and $\mathcal{F}$ be open cover of $W.$ Then there exist family $\cu\se \tau$ such that 
			$$
			\mathcal{F} = \{ U\cap W: U\in \cu\}
			$$
			and then 
			$$
			\cu^+ = \{ W^c\} \cup \cu
			$$
			forms an open cover of $X.$ Then there exists $U\in \cu^+$ and $n_0\in\w$ s.t. if $n>n_0$ then $f^n[\{ a,b\}]\in U.$ But $f[W]\se W$ and $a,b\in W$ then $U\in \cu$ and then
			$$
			(\forall n\in\w)\; (n> n_0 \then f^n[\{a,b\}] \se U\cap W \in \mathcal{F}).
			$$
			But by assumption $[y]$ is $\beta$-Hausdorff space for some $\beta<\alpha$ and $W\se [y].$ Then by the above Lemma \ref{funy clases inclusion} $W$ is a $\beta$-Hausdorff space also. Then aplying Claims \ref{claim-1-alfa} and \ref{claim-2-alfa} to the continuous contraction $f\restriction W$ on space the $W$ there is a  $x\in W$ such that $x=f(x).$
			
			Now assume that for distinct $x,y\in X$ we have 
			$$
			f(x) = x \ne y = f(x).
			$$
			Then $\cu = \{ \{x\}^c, \{y\}^c \}$ is open cover of $X.$ We known that $f$ is a contraction then
			$$
			(\forall^\infty n)\; (\{ f^n(x),f^n(y)\} \se \{x\}^c \lor \{ f^n(x),f^n(y)\} \se \{y\}^c).
			$$
			Finally we have
			$$
			\{ x,y\} \se \{x\}^c \lor \{ x,y\} \se \{x\}^c,
			$$
			what is impossible. Then there exists a unique $x\in X$ s.t. $x = f(x).$
    \end{proof}
    We have completed this proof.
\end{proof}
	
\subsection{Products of peripherally Hausdorff spaces.}
Tychonoff Theorem is a milestone in topology. Now we consider products of peripherally Hausdorff spaces. We show that arbitrary topological products of commonly finite rank of peripherally Hausdorff spaces is peripherally Hausdorff space also. Theorem do not remain true if we drop the common finite bound of ranks of countable family of peripherally Hausdorff spaces.

\begin{theorem}\label{product_of_spaces} Let $\ch$ be nonempty family of peripherally Hausdorff spaces s.t. exists $N\in \w$ s.t.
for every $X\in\ch$ $rank(X) \le N$ Then $W=\prod_{X\in\ch} X$ with product topology is peripherally Hausdorff and
$$
rank(W) = \max\{ rank(X): X\in\ch\}.
$$
\end{theorem}
\begin{proof} We prove this theorem by induction with respect to $rank(\ch)$ as $\max\{rank(X): X\in \ch\}.$ If $rank(\ch) = 0$ then each member of $\ch$ is a singleton then $W = \prod \ch  = \prod_{X\in\ch} X$ is also singleton and thus $W$ is $0$-Hausdorff space. 

Now let $N = rank(\ch) > 0.$ Firstly, we show the Claim.
\begin{claim}
For any $x\in W$ we have
$$
[x]_W = \prod_{X\in\ch} [\pi_X(x)]_X.
$$
\end{claim}
\begin{proof} To do, let $y$ be arbitrary s.t. $y \in [x]_W.$ then for every two open sets $U,V\in \tau(W)$ s.t. $x\in V$ and $y\in U$ $U\cap V \ne \emptyset.$ Fix any $X\in\ch$ and choose any open sets $U_X,V_X\in \tau(X)$ s.t. $\pi_X(x)\in V_X$ and $\pi_X(y)\in U_X$. Then $\pi_X^{-1}[U_x]$ is a open set containing $y$ and $\pi_X^{-1}[V]$ is open set in $W$ containing $x.$ Then $\pi_X^{-1}[U_x]\cap \pi_X^{-1}[V] \ne \emptyset.$ Then $U_X\cap V_Y \ne \emptyset$. But $X$ and $U_X,V_x$ are chosen in arbitrary way. Then we have established that
$$
(\forall X\in\ch)\; (\pi_X(y) \in [\pi_X(x)]_X).
$$
Then $y\in \prod_{X\in\ch} [\pi_X(x)]_X.$

Now let $y\in \prod_{X\in\ch} [\pi_X(x)]_X.$ We have to show that $y\in [x]_W.$ Let us consider two basic open sets $U,V \in \tau(W)$ s.t. $x\in V$ and $y\in Y.$ Then there are two finite families $\cf_x,\cf_y\in [\ch]^{<\w}$ and
$$
\cu = \{U_X: U_X\in \tau(X) \land  X\in \cf_y\},\;\; \cv = \{V_X: V_X\in \tau(X)  X\in \cf_x\}
$$
s.t. 
$$
\forall t\; t\in U \iff \forall X\in \cf_y\; \pi_X(t) \in U_X\;\;\land\;\;
\forall t\; t\in V \iff \forall X\in \cf_x \; \pi_X(t) \in U_X.
$$ 
To finish we have to show that $U\cap V \ne \emptyset,$ i.e. for each $X\in\ch$ $\pi_X[U] \cap \pi_X[V] \ne \emptyset.$ Let $X\in \ch$ then if $X\notin \cf_x\cup\cf_y$ then $\pi_X[U] = \pi_X[V] = X = \pi_X[U] \cap \pi_X[V]\ne \emptyset.$ If $X\in \cf_x\setminus\cf_y$ then $\pi_X[U] \cap \pi_X[V] = \pi_X[V] = V_X\ne \emptyset.$ Analogously for $X\in \cf_y\setminus\cf_x$ we have $\pi_X[U] \cap \pi_X[V] = \pi_X[U] = U_X\ne\emptyset.$ Finally, If $X\in\cf_x\cap\cf_y$ we have
$$
\pi_X[U] \cap \pi_X[V] = U_X \cap  V_X\ne \emptyset.
$$
The intersection $U_X\cap V_X$ is non empty because $\pi_X(y)\in [\pi_X(x)],$ $\pi_X(y) \in U_x$ and $\pi_X(x) \in V_X.$
    
\end{proof} 

Because each $X\in\ch$ has $rank(X_ \le N.$ Then for each $x\in W$ and for each $X\in \ch$ $rank([\pi_X(x)]_X) < N.$ Then by the above Claim we have
$$
rank([x]_W) = rank(\prod_{X\in\ch} [\pi_X(x))=\max\{ rank(\pi_X(x)):\; X\in\ch\}< N.
$$
But $x\in W$ was chosen in arbitrary way then $rank(W) \le N.$ Proof of this theorem is completed.
\end{proof}

This theorem does not have generalization to infinite products whenever rank function is unbounded in $\w$. In fact, a product of countably many peripherally Hausdorff spaces may not be peripherally Hausdorff. 

\begin{example} Indeed, if $\textrm{rank}_{T_2}(X_n)=n$ for $n\in\mathbb{N}$, then  the product of $X_n$s is not  peripherally Hausdorff. Striving for a contradiction assume that   
$$\textrm{rank}_{T_2}\left(\prod_{n}X_n\right)=\alpha$$
for some $\alpha\in\textrm{On}$. For $n\in\mathbb{N}$, let $x^{(1)}_n\in X_n$ and $\textrm{rank}_{T_2}([x^{(1)}_n]_{X_n})=n-1$. We have
$$[(x^{(1)}_n)_n]_{\prod_{n}X_n}=\prod_{n}[x^{(1)}_n]_{X_n}$$
and
$$\textrm{rank}_{T_2}([(x^{(1)}_n)_n]_{\prod_{n}X_n})=\alpha_1<\alpha.$$ 
Let now $x^{(2)}_n\in[x^{(1)}_n]$ and 
$$\textrm{rank}_{T_2}([x^{(2)}_1]_{[x^{(1)}_1]})=0,\;\textrm{and}\;\textrm{rank}_{T_2}([x^{(2)}_m]_{[x^{(1)}_m]})=m-2, \textrm{for}\;m\geq 2.$$
We have
$$\textrm{rank}_{T_2}([(x^{(2)}_n)_n]_{\prod_{n}[x^{(1)}_n]})=\alpha_2<\alpha_1.$$  
Continuing in that way we obtain a sequence of points $((x^{(k)}_n)_n)_k$, such that 
$$x^{(k)}_n\in [x^{(k-1)}_n],$$
$$\textrm{rank}_{T_2}([(x^{(k)}_n)_n]_{\prod_{n}[x^{(k-1)}_n]})=\alpha_k<\alpha_{k-1}$$
and
$$\textrm{rank}_{T_2}([x^{(k)}_1]_{[x^{(k-1)}_1]})=\ldots=\textrm{rank}_{T_2}([x^{(k)}_{k-1}]_{[x^{(k-1)}_{k-1}]})=0,$$
and
$$\textrm{rank}_{T_2}([x^{(2)}_m]_{[x^{(2)}_m]})=m-k, \textrm{for}\;m\geq k.$$
Hence none of the sets $[(x^{(k)}_n)_n]_{\prod_{n}[x^{(k-1)}_n]}$ has the Hausdorff rank equal to zero and we obtain the infinite decreasing sequence of ordinals:
$$\alpha>\alpha_1>\alpha_2>\ldots,$$
which is impossible.
\end{example}

The above example tells us that the infinite product of peripherally Hausdorff spaces can be not peripherally Hausdorff space. However, a finite product of any peripherally Hausdorff spaces is peripherally Hausdorff space.

\begin{theorem} Product of two peripherally Hausdorff spaces is a peripherally Hausdorff space. Moreover, if $rank(X) = \alpha$ and $rank(Y) =\beta$ then 
$$
rank(X\times Y) = \max\{\alpha,\beta\}.
$$
\end{theorem}
\begin{proof} Induction over rank of peripherally Hausdorff spaces. If $X$ and $Y$ are singletons then $X\times Y$ is also so theorem is true. Let $\alpha$ be any positive ordinal number and theorem is true for any two peripherally Hausdorff spaces with rank less than $\alpha.$ Let $rank(X) = \alpha$ and $rank(Y)\le \alpha.$ Choose any $(x,y) \in X\times Y.$ By similar argument as in proof of the Theorem \ref{product_of_spaces} we see that 
$$
[(x,y)]_{X\times Y} = [x]_X\times [y]_Y
$$
Because spaces $X, Y$ have rank $\alpha,\beta$ with $\beta\le \alpha.$ then spaces  $[x]_X, [y]_Y$ have rank less than $\alpha$ and are proper subspaces of $X$ and $Y$ respectively. By induction assumption we have
$$
rank([(x,y)]_{X\times Y}) = \max\{ rank [x]_X,rank [y]_Y\} < \alpha = rank (X)
$$
and $[x]_X\times [y]_Y$ is proper subspace of space $X\times Y.$ Then $rank (X\times Y) \le \alpha.$ Let observe that $rank (X\times Y) = \alpha$, indeed if no then there is $\gamma <\alpha$ s.t. for every $(x,y)\in X\times Y$ we have $rank [(x,y)]_{X\times Y} < \gamma.$ But by definition of $rank (X)$ there is $x_0\in X$ s.t. $\gamma \le rank [x_0]_X< \alpha.$ But for every $y\in Y$ we have
$$
\gamma \le rank [x_0]_X \le rank [x_0]_X\times \{y\}\le rank [x_0]_X\times [y]_Y = rank [(x,y)]_{X\times Y} < \gamma,
$$
contradiction. Finally, the space $X\times Y$ is peripherally Hausdorff space with rank equal to $\alpha = \max\{rank (X), rank (Y)\}.$

The proof is exactly the same if we assume that $rank (X) \le rank (Y) = \alpha.$
\end{proof}
\subsection{Tree structure of a peripherally Hausdorff spaces}
In this section we every peripherally Hausdorff space encode by well founded tree. By the tree $T=(V,E,r)$ we consider any directed connected tree without cycles which has distinguished  element $r\in V$ so called root of tree. Root of tree is unique vertex (node) with property that for every node $x\in V$ there is a chain $(a_i)_{i<n} \in V^n$ (path from $r$ to $x$) s.t. 
\begin{itemize}
    \item $(\forall i<n)\; (i+1<n\then (a_i,a_{+1}) \in E),$
    \item $a_0 = r$ and $a_{n-1} = x.$
\end{itemize}
For any $x,y\in V$ we say that $x\prec y$ if there is a path from $x$ to $y.$ It is easy to see that $P_T = (V,\prec)$ forms a partial order. We say that tree $T$ is well founded if there is no infinite chains in $P_T.$ It is well known that with every well-founded tree $T=(V,E,r)$ we can associate the so called rank function $rk:V_T\to On$ s.t for every $x,y\in V$ if $(x,y)\in E$ then $rk(y) < rk (x).$

Moreover, for every tree $T=(V_T,E_T,r_T)$ we can define $Rank(T)$ as follows:
$$
Rank(T) = \inf\{ \alpha\in On:\; (\exists rk)\; rk \text{ is a rank function on } T\land \alpha = rk(r_T) \}.
$$

Now for every peripherally Hausdorff space $X$ we find a tree $T_X$ and $rk:V_T\to On$ rank function that the following condition is holds
$$
rank(X) = Rank(T_X = rk (r_T).
$$
We can do this by the reccursion with respect of $rank(X).$ Let $rank(X) = 0$ then $X = \{x\}$ and then  set $V_X = \{X\},$ $E_X=\emptyset$ and $r_X = \{x\}.$ Then $T_X=(V_X,E_X,r_X)$ is well founded tree of course. Let us define $rk:V_X\to On$ as $rk(x) = 0.$ Of course, $rank(X) = rk(x) = Rank(T_X).$

Now let $\alpha>0$ be ordinal s.t. $rank(X) = \alpha$ and for every $\beta< \alpha$ and every $Y$ with $rank(Y) = \beta$ there are tree $T_Y$ and $rk:V_T\to On$ rank function on $T_Y$ s. that
$$
\beta = rank(Y) = rk(r_{T_Y}) = Rank(T_Y).
$$
Then set $T_X=(V_X,E_X,r_X)$ as an amalgamation of trees $\ct = \{T_{[x]}:\; x\in X\}$ with the root $r_T = X.$ where for each $T\in\ct$ there is $rk_T:V_T\to On$ s.t. 
\begin{enumerate}
    \item $T = T_{[x]}$ for some $x\in X,$
    \item $rk(r_T) = Rank(T) = rank([x]) < \alpha.$
\end{enumerate}
Then define function $rk_{T_X}:V_X\to On$ as follows:
$$
rk = \bigcup\{rk_T: T\in\ct \} \cup \{(r_X,\delta)\}
$$
where
$$
\delta = \sup \{ Rank(T)+1:\; T\in \ct\}.
$$
It is easy to see that $\delta = \alpha$ and $\alpha = rank(X) = rk(r_{T_X}) = Rank(T_X).$

We consider Niemytzki half-plane Example \ref{Niemyrzki_half_plane}.

\begin{example}\label{Niemytzki_tree} Let $X$ be  Niemytzki half-plane. Then using recurrsion explained above we can construct $T_{X} = (V_X,E_X,r_X)$ as follows: $V_X = V_2 \cup V_1 \cup V_0,$ where $V_2 = \{ X\} = \{ r_X\},$ $r_X$ is a root of $T_X,$
$V_1 = V_{10} \cup V_{11},$
$$
V_{10} = \{ \{(x,y)\}\cup \{(x,0)\}:x\in\bbr\land y>0\},
$$
$$
V_{11} = \{ \bbr\times\{0\}\cup \{x\}\times (0,\infty): x\in\bbr \land y = 0 \},
$$

$V_0 = W_{00}\cup W_{10},$
\begin{align*}
W_{00} = & \big\{ [(a,b)]_{ \{(x,y)\}\cup \{(x,0)\} }:\; x\in\bbr\land y>0 \land (a,b) \in \{(x,y)\}\cup \{(x,0)\} \big\}\\
= & \big\{ \{(a,b)\}:\; (\exists x\in\bbr)(\exists y>0)\; (a,b) \in \{(x,y)\}\cup \{(x,0)\} \big\},    
\end{align*}
$$
W_{01} = \{ \{(x,0)\}: x\in\bbr \} \cup \{ \{(x,y)\}:\; x\in\bbr \land y>0\}.
$$

Finally, for each $v\in V_X$ and each $i\in\{0,1,2\},$ we define $rk(v) = i \iff v \in V_i.$ Then $RankT_X) = rank(X) = 2.$

Here we can present out tree $T_X$ as follows.

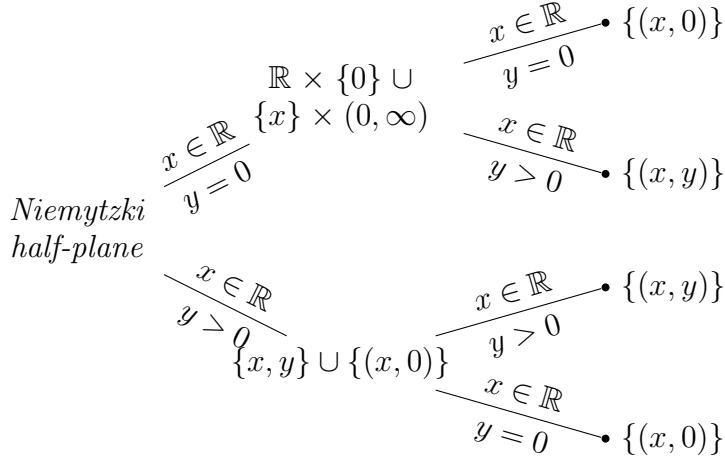
\begin{tikzpicture}[grow=right, sloped]
\node[node] {Niemytzki half-plane}
    child {
        node[node] {$\{x,y\}\cup \{(x,0)\}$}        
            child {
                node[leaf, label=right:
                    {$\{(x,0)\}$}] {}
                edge from parent
                node[above] {$x\in\bbr$}
                node[below]  {$y = 0$}
            }
            child {
                node[leaf, label=right:
                    {$\{ (x,y) \}$}] {}
                edge from parent
                node[above] {$x\in\bbr$}
                node[below]  {$y>0$}
            }
            edge from parent 
            node[above] {$x\in \bbr$}
            node[below]  {$y>0$}
    }
    child {
        node[node] {$\bbr\times\{0\}\cup \{x\}\times (0,\infty)$}        
        child {
                node[leaf, label=right:
                    {$\{(x,y)\}$}] {}
                edge from parent
                node[above] {$x\in\bbr$}
                node[below]  {$y>0$}
            }
            child {
                node[leaf, label=right:
                    {$\{(x,0)\}$}] {}
                edge from parent
                node[above] {$x\in\bbr$}
                node[below]  {$y = 0$}
            }
        edge from parent         
            node[above] {$x\in\bbr$}
            node[below]  {$y = 0$}
    };
\end{tikzpicture}
\end{example}

Now we pass to example borrowed from proof of the Theorem \ref{H_alpha}.
\begin{example} Let $Y$ is a peripherally Hausdorff space with rkan equal to $\beta.$ We follow of construction of the space $X$ with the $rank(X) = \alpha = \beta+1.$ Set $X= Y\cup\{a_n:n\in\w\},$ where
$Y\cap \{a_n:n\in\w\} = \emptyset,$ Now we define a base at each $x\in X.$ Namely, if $x\in Y$ then every $U=U_X\cup A\in\cb_x$ if $x\in U_X,$ $U_X\se Y$ is open in $Y$ and $A\se \{y_n:n\in\w,\}$ is co-finite in $\{a_n:n\in\w\}.$ If $x = a_n$ for some $n\in\w$ then $V$ is open neighbourhood of $x$ if $x\in V\se \{a_n:n\in\w\}$ and $V$ is finite in $\{a_n:n\in\w\}.$ 

Let observe that for every $y,y'\in Y$ and any $U\in \cb_y$ $V\in \cb_{y'}$ $U\cap V$ contains of some co-finite subset of $\{a_n:n\in\w\}.$ But for every $n\in\w$ and $y\in Y$ there two disjoint open neighborhood of $y$ and $a_n.$ Thus for every $x\in X$ we have
$$
[x]_X = 
\begin{cases}
    Y       & \text{ for } x \in Y\\
    \{x\}   & \text{ for } x \in \{a_n:n\in\w\}.
\end{cases}
$$
Here $T_X$ is amalgamation of family of trees 
$$
\{ T_Y\}\cup \{T_{\{a_n\}}:\; n\in \w\}
$$
at root $r_T = X.$

Of course by definition $rk(T_Y) = \beta$ and $rk(T_{\{y_n\}}) = 0$ for $n\in \w.$ Then by the definition of $rk$ function on $T_X$

\begin{align*}
rk(r_X) =  rk (T_X) = & \max \{rk(T_Y)+1, \sup \{ rk(T_{\{a_n\}})+1:n\in\w \} \}\\ 
= & \max\{\beta+1, 0+1\} =  \beta+1 = \alpha.    
\end{align*}

Now we can draw  a tree $T_X$.

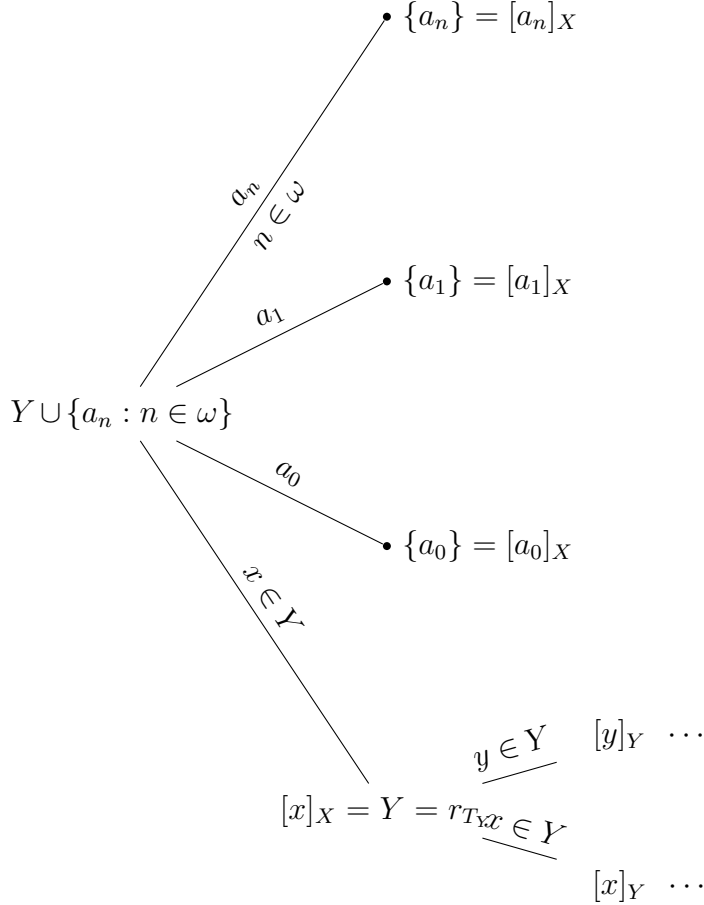
\begin{tikzpicture}[grow=right, sloped]
\node[node] {$Y\cup \{a_n:n\in\w\}$}
    child {
        node[node] {$[x]_X = Y = r_{T_Y}$}        
            child {
                node[node]
                    {$[x]_Y \; \cdots$}
                edge from parent
                node[above] {$x\in Y$}
                node[below]  {}
            }
            child {
                node[node]
                    {$[y]_Y \; \cdots$}
                edge from parent
                node[above] {$y\in Y$}
                node[below]  {}
            }
            edge from parent 
            node[above] {$x\in Y$}
            node[below]  {}
    }
    child {
        node[leaf, label=right: {$\{a_0\} = [a_0]_X$}] {}        
        edge from parent         
            node[above] {$a_0$}
            node[below]  {}
    }
    child {
        node[leaf, label=right: {$\{a_1\} = [a_1]_X $}] {}        
        edge from parent         
            node[above] {$a_1$}
            node[below]  {}
    }
        child {
        node[leaf, label=right: {$\{a_n\} = [a_n]_X$}] {}        
        edge from parent         
            node[above] {$a_n$}
            node[below]  {$n\in\w $}
    }
    ;
\end{tikzpicture}
\end{example}

\section{Locally Hausdorff spaces}
In this section, we will investigate the so called locally Hausdorff spaces. In this class of space the fixed point theorem holds but for continuous feebly contraction. 

By Kupka Theorem, every continuous weak contraction on Hausdorff space has a unique fixed point. Moreover, in every first countable Hausdorff space, the fixed point theorem is true for any feebly topological contraction which is closed map\footnote{By closed map we mean that every image of closed set is closed.} also (we do not require continuity). We show that in first countably locally Hausdorff spaces the mentioned theorem can be not true.

\begin{definition}[Locally Hausdorff space] We say that any $T_1$ topological space $X$ is locally Hausdorff if for every point there is an open neighborhood $U$ such that $U$ is Hausdorff space in relative topology.
\end{definition}

Every Hausdorff is, of course, a locally Hausdorff space. The space in Example \ref{no_closed_graph} is a Hausdorff space locally, but not a Hausdorff space.

Now we are ready to formulate the main theorem of this section.
\begin{theorem}\cite[Theorem 6.1]{MR1} Every continuous feeble topological contraction has a unique fixed-point in locally Hausdorff space.
\end{theorem}
\begin{proof} Let us assume that the conclusion does not hold. Let $x\in X.$

First of all, we find a point y such that for every open neighbourhood $U$ of $y$ there exists $n\in\bbn$ such that $\{fn(x), f n+1(x)\} \se U.$ If not then there is an open cover $\cu$ of $X$ s.t. there is no $n\in\bbn$ that $\{f^n(x),f^{n+1}\}\nsubseteq U$ for every $U\in\cu.$ But this contradicts with $f$ is feeble contraction.

Now we show that for any $m\in\bbn$ and any open sets $U,V$ containing $f^m(y), f^{m+1}(y)$ respectively we have $U\cap V\ne \emptyset.$ Indeed, let observe that
$$
W = \underbrace{f^{-1}[\ldots f^{-1}[}_{m} U]\ldots ]\cap 
\underbrace{f^{-1}[\ldots f^{-1}[}_{m+1}V]\ldots ]
$$
is neighbourhood of $y$ and then $\{f^k(x),f^{k+1}(x)\}\se W$ for some $k\in\bbn.$ Thus we have $f^{m+k+1}(x)\in U\cap V.$

Because $X$ is locally Hausdorff then there is open cover $\cu$ of $X$ s.t. each member of $\cu$ is $T_2$ open set. But $f$ is feeble contraction then there are $m\in\bbn$ and $U\in\cu$ s.t. $\{f^m(y),f^{m+1}(y)\}\se U.$ Then $f^m(y),f^{m+1}(y)$ have disjoint neighbourhoods, what is impossible by the previous remark.

Uniques of fixed point follows in the same way like in he proof of the Theorem \ref{funy fixed point}.
\end{proof}

Of course we have.
\begin{cor} Every continuous weak topological contraction on Hausdorff space has a unique fixed-point.
\end{cor}

Here we present proof of the Theorem about existence of fixed point for closed feeble contraction defined on Hausdorff space.

\begin{theorem}\cite[Theorem 4.5]{MR1} Every closed feeble topological contraction has a unique fixed-point in first countable Hausdorff space.
\end{theorem}

Here we present example where this theorem can not be formulated on every first countable locally Hausdorff space.
\begin{example} Let $X=\bbz$ with subbase:
$$
\{ n\cup A:\ (n\in\bbz\setminus\w\then  (A\se \w\land \w\setminus A\in [w]^{<\w}))\ \land (n\in\w \then A=\emptyset)\}.
$$
Let observe that $X$ is first countable and $\{n\}\cup \w$ is Hausdorff subspace for any $n<0.$ Now define $f(n) = n+1$ for every $n\in X.$ It is easy to check that $f$ is feeble contraction. Moreover, $f$ is closed map because is bijection on $X$ and image of any open subset is  open. Let observe that $f$ has not a fixed point and $f^{-1}[\{0\}] = \{-1\}$ which is not open and then $f$ is not continuous map.
\end{example}

\section{Tree spaces}

In this section, we consider a wider class of topological spaces than peripherally Hausdorff spaces. The price is that the fixed point requires that the ultra-fine space should be compact. We do not know whether compactness can be dropped.

The inspiration for this  kind of space was the tree structure of a peripherally Hausdorff space. Trees for these spaces were well-founded. Here we use a tree with countable paths (chains).

Fix cardinal $\kappa,$ then $T\se \kappa^{<\w}$ ids tree iff
$$
(\forall \sigma\in T)(\forall n\in\w)\ (\sigma\restriction n \in T).
$$
For the tree $T\se \kappa^{<\w}$, the set $[T]=\{ t\in \kappa^\w:(\forall n\in\w)\ (t\restriction n\in T)\}$ is called the body of $T.$ Next, if $\sigma \in T$ then
$succ_T(\sigma)=\{\alpha<\kappa:\ \sigma\with\alpha\in T\}$ denotes all successors of $\sigma$ in $T.$

\begin{definition}[Tree space] Let $X$ be a topological space and $X=\{x_\alpha:\alpha<\kappa\}$ be its enumeration. We say that the $(X,T,\rho)$ is an tree space if
\begin{enumerate}
    \item $T\se\kappa^{<\w}$ is a tree and $\emptyset\in T$,
    \item $\rho:T\to P(X)$ s.t. $\rho(\emptyset) = X,$\\
    for every $\sigma\in T$ $succ_T(\sigma) = \{ \alpha<\kappa:\ x_\alpha\in \rho(\sigma)\}$ and 
    $$
    (\forall \alpha\in succ_T(\sigma))\  (\rho(\sigma\with \alpha) = [x_\alpha]_{\rho(\sigma)}),
    $$
    \item $\forall t\in[T]$ $\bigcap\{ \rho(t\restriction n): n\in\w\}$ is a singleton,
\end{enumerate}
\end{definition}

Before the fixed-point theorem, we show the following lemma.
\begin{lemma}\label{lemma-ultra} If $(X, \tau)$ is a topological space and $f:X\to X$ is a feeble$+$ topological contraction, then
$$
(\forall x\in X)(\exists y\in X) (\forall U_y\in \tau) (\forall n^\infty)\ f^n(x) \in U_y,
$$
where $f^n$ is $n$-th iteration of $f,$ $\tau$ topology and $y\in U_y.$ 
\end{lemma}
\begin{proof} If not, then for some $x\in Y$ and every $y\in Y$, we can find $U_y$ such that for infinitely many $n\in\w$ $f^n(x)\notin U_y.$ So $\{U_y:y\in X\}$ is an open cover of $X.$ Then, by the weak$^{+}$ contractivity of $f$, there exists $y_0\in X$ such that $\forall n^\infty f^n[\{x,f(x)\}]\se U_{y_0}.$ Then, $\forall n^\infty f^n(x)\in U_{y_0},$ is a contradiction.
\end{proof}

\begin{theorem} If $(X,T,\rho)$ is a compact tree space and $f:X\to X$ is a continuous feeble$+$-contraction the there exists unique $x\in X$ s.t. $x = f(x).$
\end{theorem}
\begin{proof} By the recursion, we define the sequence $\{\sigma_n\in T: n\in\w\}$, $\{y_n:n\in\w\}$, and $\{ F_n:\ n\in\w\}$ such that $\sigma_0=\emptyset,$ $y_0\in X,$ $F_0=X$ and for every $n\in\w$  the following is holds
\begin{enumerate}
    \item $\sigma_n\subsetneq \sigma_{n+1}$,
    \item $F_{n+1}\se F_n$ and $f[F_n]\se F_n,$
    \item $F_n = \clo{ \{ f^k(y_n):k\in\w\} } \se [y_n]_{\rho(\sigma_n)}\se \rho(\sigma_n).$
\end{enumerate}
Let us assume that $\sigma_n\in T$ and $y_n\in F_n$ are defined with 
$$
F_n = \clo{ \{ f^k(y_n):k\in\w\} } \se [y_n]_{\rho(\sigma_n)}\se \rho(\sigma_n).
$$

Let us observe that by the continuity of $f$, we have $f[F_n]\se F_n$, and $f\restriction F_n$ is a continuous feeble$^{+}$-contraction on $F_n.$
 Then, using the Lemma \ref{lemma-ultra} for $X = F_n,$ and $x = y_n$, we can find $y = y_{n+1}\in F_n\se [y_n]_{F_n}\se  [y_n]_{\rho(\sigma_n)}$ such that.
$$
\forall U\in \tau_{y_{n+1}}\forall k^\infty f^k(y_n)\in U.
$$
As in the proof of the fixed point theorem for peripheral spaces, using the continuity of $f$, we have the following Claim.
\begin{claim}
$$
\forall k\in\w f^k(y_{n+1}) \in [y_{n+1}]_{F_n}.
$$    
\end{claim}
\begin{proof}[Proof of Claim] Let us observe that $f^0(y_{n+1}) = y_{n+1} \in [y_{n+1}].$ Let us assume that for some $k\in\w$, we have $f^k(y_{n+1})\in [y_{n+1}]_{F_n}$, but $f^{k+1}(y_{n+1})\notin [y_{n+1}]_{F_n}.$ By the definition of $[y_{n+1}]_{F_n}$, we have that for every $U\in\tau_{y_{n+1}}(F_n)$ $y_{n+1},f^k(y_{n+1})\in \clo{U}.$ But $f^{k+1}(y_{n+1})\notin[y_{n+1}]_{F_n}$ then there is $U\in\tau_{y_{n+1}}(F_n)$ s.t. $f^{k+1}(y_{n+1})\notin \clo{U}.$ So, the complement $V=\clo{U}^c$ in space $F_n\se \rho(\sigma)$ is an open neighborhood of $f^{k+1}(y_{n+1})$ in $F_n.$ But it is not possible because 
$$
\forall j^\infty f^i(y_n) \in U \land \forall j^\infty f^i(y_n) \in V,
$$
what is not true because $U\cap V = \emptyset.$
\end{proof}

Then there exists $\alpha<\kappa$ such that $y_{n+1} = x_\alpha\in F_n\se \rho(\sigma_n).$ Set $\sigma_{n+1} = \sigma\with \alpha$ and then $\rho(\sigma_{n+1}) =  [y_{n+1}]_{\rho(\sigma_n)}.$ 

Set $F_{n+1} = \clo{\{f^k(y_{n+1}): k\in\w\}}.$ Then $F_{n+1} \se \rho(\sigma_{n+1})$ and $F_{n+1}\se F_n.$ Moreover, we have $f[F_{n+1}] \se F_{n+1}.$ The $n+1$ step is finished.

Then, by the compactness of $X$, we have 
$$
\emptyset\ne \bigcap\{F_n:n\in\w\}\se \bigcap\{ \rho(\sigma_n): n\in\w\} = \{x\},
$$
for some $x\in X.$ Then $\{x\} = \bigcap\{F_n:n\in\w\}.$ But for every $n\in\w$ $f[F_n]\se F_n$ then $f(x)\in F_n$ for each $n\in\w.$ Finally, $f(x) = x.$ Using feeble$^{+}$ contractivity of $f$ we can conclude that there exists a unique fixed point of $f.$ 
\end{proof}

\section{\v{C}ech complete spaces, atractors on $T_1$ compact spaces}
In the previous sections, all functions should be continuous weak or feeble$+$-contractions in fixed-point theorems.  Now we prove the fixed-point theorem in the so called weak \v{C}ech complete spaces. For the original \v{C}ech complete spaces topology, see \cite{E}. But the price of it is that our functions should be closed maps instead of continuous functions.

The notion of weak \v{C}ech completeness is a generalization of a metric space that is complete. The notion of \v{C}ech complete space is well known, but we drop the assumption that our space is Tychonoff.

\begin{definition} 	Topological $T_1$ space $X$ is weak \v{C}ech complete if
	\begin{itemize}
		\item there exists $ \{\cu_i: i\in\w \}$, $\cu_i$ - an open cover of $X$ for  $i\in\w$,
		\item for every centered $\{F_m\in Clo(X):m\in\w\}$ such that $\forall i\in\w\, \exists m\in\w\, \exists U\in\cu_i\; F_m\subseteq U$ 
	\end{itemize}
	then $\bigcap \{F_m\:\; m\in\w\} \ne \emptyset.$
\end{definition}
Let us observe that every complete metrizable space is a weak \v{C}ech complete space. Every compact space is also \v{C}ech complete space.

Here we introduce the notion of the topological contraction, which we use in the fixed-point theorem in weak \v{C}ech complete spaces.
\begin{definition}[Topological contraction] Let $X$ be a topological space. Then $f:X\to X$ is a topological contraction if for every open cover $\cu$ of the space $X$, there exists $n\in\w$ such that $f^n [X]\se U$ for some $U\in\cu.$
\end{definition}

\begin{theorem}\cite[Theorem 3.2]{MR1}\label{Banach-T1-complete}
	If $X$ is a $T_1$ weak \v{C}ech complete space and $f:X\to X$ is a closed topological contraction, then $f$ has a unique fixed point.
\end{theorem}

\begin{proof}  Let $(\mathcal{U}_i)_i$ be a sequence of open covers of $X$ from the definition of \v{C}ech completeness. The sequence $(f^n[X])_{n\in\w}$ is a non-increasing sequence of non-empty closed sets ($f$ is a closed mapping); thus, it is centered. Because $f$ is a topological contraction, for every $i\in\mathbb{N}$, there exists $n_i$ such that $f^{n_i}[X]$ is contained in some element of $\mathcal{U}_i$. Hence, because of the choice of the sequence of open covers $(\mathcal{U}_i)_i$,   the intersection
$\bigcap_nf^n[X]$ is nonempty. 

Now let $x\in\bigcap_nf^n[X]$. Because $f(x)\in\bigcap_nf^n[X]$, it follows that $f(x)=x$ and $x$ is a fixed point of $f$. 

Now we show the uniqueness of the fixed-point. Let us assume that there are two distinct fixed-points  $x,y\in X.$ Let us consider the open cover $\cu=\{ X\setminus\{x\},X\setminus\{y\}\}.$ Then there exists $n\in\w$ such that $f^n[X]\se X\setminus \{x\}$ or $f^n[X]\se X\setminus \{y\}.$. Let us assume that $f^n[X]\se X\setminus \{x\}$ for some $n\in\w.$ Observe that 
$$
x = f^n(x) \in f^n[X]\se X\setminus \{x\},
$$
which is impossible. The second case in which we get a contradiction is analogous.

\end{proof}

Because every compact space is a weak \v{C}ech complete space, we have the following corollary. 
\begin{theorem}\cite[Theorem 3]{MR2}\label{compact-fix} IF $X$ is $T_1$ compact, $f:X\to X$ closed topological contraction, then $f$ has a unique fixed point.
\end{theorem}

Here, we give an example of a closed topological contraction that is not continuous.
\begin{example} Let $X=\{0,2,3\}\cup\{\frac{1}{n+1}: n\in\w\}$ with euclidean topology. Let us define $f:X\to X$ as follows: for any $x\in X$ set
$$
f(x) = \begin{cases}
    2 & x = \frac{1}{n+1} \text{ for some } n\in\w,\\
    3 & x\in \{0,2,3\}.
\end{cases}
$$
This function is not continuous at point $x = 0.$. Let us observe that $f[X]\se \{2,3\}$; then $f$ is closed. Because $f^2[X] = \{3\},$ then $f$ is a topological contraction.
\end{example}

Here we present the equivalence of the topological contraction with a much simpler condition in the case of compact $T_1$ spaces.

\begin{proposition}\label{proposition-contraction} Let $f:X\to X$ be a closed map on a compact $T_1$ space $X.$ Then the following conditions are equivalent:
\begin{enumerate}
    \item $f$ is a topological contraction on $X,$
    \item $(\forall x,y\in X)(\exists n\in\w)\; (f^n[X]\se X\setminus \{x\} \lor f^n[X]\se X\setminus \{y\}).$
\end{enumerate}
\end{proposition}
\begin{proof} The implication (1) to (2) is true by the definition of the topological contraction.

Then we show the other direction. Let us assume that (2) is true but (1) does not hold. By the fixed point Theorem \ref{compact-fix}, there exists a unique $x\in X$ such that $x = f(x).$ Let $\cu$ be any open cover of $X$ and $U\in\cu$ with $x\in U.$ We show that there exists $n\in\w$ such that $f^n[X]\se U.$ If not, then by the compactness of $X$ and taking into account that $f$ is a closed map, we have $U^c\cap \bigcap\{f^m[X]:n\in\w\}\ne \emptyset.$ Then let us choose any $y\in (\bigcap\{f^n[X]:n\in\w\})\setminus U.$ Then for every $n\in\w$, $f^m[X]$ is not contained in $X\setminus\{x\}$ and $f^n[X]$ is not contained in $X\setminus\{y\},$ which is impossible by $(2)$ condition.
\end{proof}

The fixed-point theorem \ref{compact-fix} is an entry point to attractors in $T_1$ compact spaces with respect to the closed IFS system.

\subsection{IFS and its Hutchinson atractors}
The Hutchinson operator for fixed IFS (iterated function systems) is mainly considered from a complete metric point of view  see \cite{LSS, Ba} and \cite{H} for example. In this context, the Banach fixed-point Theorem plays a crucial role. We define the IFS system on $T_1$ compact space analogously, and instead of the Banach Theorem, we use Theorem \ref{compact-fix} and the other one.

Iterated function systems (IFS) on topological spaces can be found in \cite{Mih} for example.

For a fixed topological space $X,$, any finite family $\cf\se X^X$ is called an IFS, i.e., an iterated function system. By $2^X\se P(X)$ we mean the space of all nonempty closed subsets (called the hyperspace of $X$) equipped with the Vietoris topology generated by the following sets
$$
V(U_0,V_1,\ldots, V_n) = \{ F\in 2^X:\; (F\se V_0) \land (\forall i\le n)\ (i>0 \then F\cap U_i\ne \emptyset)\},
$$
where $n\in\w$ and $\{ U_i:\ i\le n\}$ is a family of open subsets of $X.$
It is well known that if $X$ is a $T_1$ compact space, then $2^X$ is also $T_1$ and compact.

Now we recall the contractivity of the IFS Hutchinson operator.
\begin{definition} For a fixed topological space $X$, let $\cf\in [X^X]^{<\w}$ be an IFS and let us assume that $\cf=\{ f_i:\; i < |\cf|\}$ is its enumeration.

Then $\cf$ is a contractive IFS if, for every open cover $\cu$ of $X$, there exists $n\in\w$ such that for every sequence $s\in |\cf|^n$, there exists $U\in \cu$ such that $f_{s(0)}\circ \ldots \circ f_{s(n-1)} [X] \se U,$
, where $\circ$ is the composition of functions.
\end{definition}

\begin{definition} If $\cf$ is IFS on space $X$, the $F_\cf:2^X \to 2^X$ defined by
$$
2^X \ni A \to F_\cf(A) = \bigcup\{ f[A]: f\in\cf\}\in 2^X
$$
is called the Hutchinson operator of IFS $\cf.$ Sometimes we will drop the subscript.
\end{definition}

Now we show that contractive IFS reflects on the contractivity of the Hutchinson operator in the hyperspace.
\begin{theorem}\label{Hutchinson-contractive} If $\cf$ is a contractive IFS on a compact $T_1$ space $X$, then the Hutchinson operator $F_\cf$ is a topological contraction on the hyperspace $2^X.$
\end{theorem}
\begin{proof} Because $2^X$ is compact, it is enough to show the second condition from Proposition \ref{proposition-contraction}.

Let $K_0,K_1\in 2^X$ be arbitrary distinct closed subsets of $X.$ Without loss of generality, we can assume that there is $z \in K_1\setminus K_0.$ We have to show that at some $n\in\w$ $F[2^X]\se 2^X\setminus \{K_0\}$ or $F[2^X]\se 2^X\setminus \{K_1\}.$ If not, then for every $n\in\w$ $\{ K_0,K_1\} \se F^n[2^X].$ Then for every $n\in \w$, there are $E_n^0, E_n^1\in 2^X$ such that $K_i = F(E_n^i)$ for each $i\in\{0,1\}.$

For each sequence $s\in|\cf|^n$, set 
$$
f_s = f_{s(0)}\circ\ldots \circ f_{s(n-1)}.
$$
Let us observe that for each $n\in\w$ we have
$$
F^n(E_n^i) = \bigcup\{ f_s[E_n^i]: s\in |\cf|^n\},
$$
for $i\in \{0,1\}.$ Then for each $n\in\w$ there is $s_n\in |\cf|^n$ such that $z \in f_{s_n}[E_n^1].$

Now let us consider a two point open cover of $X$ $\cu = \{K_0^c, \{z\}^c\}.$ Then by the contractivity of $\cf$, there is $m\in\w$ such that for every $s\in |\cf|^m$ $f_s[X]\se K_0^z$ or $f_s[X]\se \{z\}^c.$ But for some $t\in |\cf|^m$, we have $z \in  f_t[E_m^1] \se f_t[X]$ and then $f_t[X]\se K_0^c.$ Now we have 
$$
f_t[E_m^0]\se f_t[X] \se K_0^c
$$
but from other hand we have 
$$
f_t[E_m^0]\se F^m(E_m^0) = K_0,
$$
what is impossible.
\end{proof}

Now we show the result regarding the existence of a fixed point for the Hutchinson operator of IFS.
\begin{theorem}\label{fix-Hutchinson}\cite[Theorem 5]{MR2} Let $\cf$ be an IFS of closed maps on a compact $T_1$ space $X.$ Then, the Hutchinson operator $F_\cf$ has a fixed point in $2^X.$
\end{theorem}
\begin{proof} We use the analogy of Bendixon rank. We define a transfinite sequence of subsets of $X$ $\{G_\alpha: \alpha\in On\}$ as follows:
\begin{enumerate}
    \item $G_0 = F_\cf(X),$
    \item if $\alpha = \beta+1$ then $G_\alpha = F_\cf(G_\beta),$
    \item if $\alpha$ is a limit ordinal, then $G_\alpha = \bigcap\{ G_\beta:\beta<\alpha\}.$
\end{enumerate}
The defined sequence above is non-increasing and consists of closed, nonempty sets. Then there is $\alpha < (2^{|X|})^+$ such that $G_{\alpha+1} = G_\alpha.$ Then $G_\alpha = G_{\alpha+1} = F_\cf(G_\alpha)$ is a fixed point of the Hutchinson operator $F_\cf.$
\end{proof}
Let us remark that in the above Theorem \ref{fix-Hutchinson}, IFS of closed maps on $X$ may not be contractive on $X.$ However, we do not have a guaranty that such a fixed-point is unique.

\begin{theorem}\cite[Theorem 2]{MR1} If $X$ is a compact $T_1$ space and $\cf$ is a contractive IFS of closed maps on $X$, then the Hutchinson operator $F_\cf$ has a unique fixed-point in $2^X.$
\end{theorem}
\begin{proof} By the previous Theorem \ref{fix-Hutchinson}, there exists $D\in 2^X$ such that $F_\ch(D) = D.$. Let us assume that there exists another $E\in 2^X$ fixed-point of $F_\cf.$ Because $\cf$ is an IFS contraction, then by Theorem \ref{Hutchinson-contractive}, $F_\cf$ is also a topological contraction on $2^X.$. Then, by the contractivity of our Hutchinson operator, there exists $n\in\w$ such that
$$
\{F^n_\cf(E), F^n_\cf(G)\} \se 2^X\setminus\{E\} \text{ or } 
\{F^n_\cf(E), F^n_\cf(G)\} \se 2^X\setminus\{G\}.
$$
But $F^n_\cd(E) = E,$ and $F^n_\cf(G) = G$ imply that the above inclusion is impossible.
\end{proof}

One can try to apply Theorem \ref{compact-fix} on compact hyperspace $2^X$ (whenever $X$ $T_1$ is compact). But contractive IFS of closed operators on $X$ do not guaranty that its Hutchinson is a closed map on $2^X.$ Here is an example.
\begin{theorem}  There exists a $T_1$ compact space $X$ and a contractive IFS inducing a Hutchinson operator that is not closed (as a mapping from $2^X$ to $2^X$).
\end{theorem}
\begin{proof} Let 
$$\textrm{ODD}=\{1,3,5,\ldots\}$$
and
$$\textrm{EVEN}=\{2,4,6,\ldots\}.$$
Let 
$$X=\mathbb{N}\cup\{a,b\}$$
where $a\neq b$, $a,b\notin\mathbb{N}$, and the set $A\subseteq X$ is open if:

\vspace{0.2 cm}

\noindent i) $A\subset \textrm{ODD}$,

\vspace{0.2 cm}

or  

\vspace{0.2 cm}

\noindent ii) $A\subseteq\mathbb{N}$ and the set $\textrm{ODD}\setminus A$ is finite,

\vspace{0.2 cm}

or  

\vspace{0.2 cm}

\noindent iii) $A\cap\{a,b\}\neq\emptyset$ and $A$ is a co-finite set in $X$.

\vspace{0.2 cm}

We define an IFS $\mathcal{F}$ as 
$$\mathcal{F}=\{f,g\},$$
where 
$$f(x)=\left\{
\begin{array}{ccc}
a&\textrm{if}&x=a,b,\\
b&\textrm{if}&x=2n,n\in\mathbb{N},\\
2n&\textrm{if}&x=2n-1,n\in\mathbb{N},\\
\end{array} 
\right.
$$
and
$$g(x)=\left\{
\begin{array}{ccc}
b&\textrm{if}&x=a,b,\\
a&\textrm{if}&x=2n,n\in\mathbb{N},\\
2n&\textrm{if}&x=2n-1,n\in\mathbb{N}.
\\
\end{array} 
\right.
$$

\vspace{0.2 cm}

Let $E\subseteq X$ be a nonempty closed set. If $E$ is finite,
then $f[E]$ is also finite and, therefore, closed. 
If $E$ is infinite, then $a,b\in E$ and $E\cap\textrm{EVEN}\neq\emptyset$.
Then $\{a,b\}\subseteq f[E]\subseteq\{a,b\}\cup\textrm{EVEN}$ which is a closed set. 

Thus $f$ is a closed mapping. Analogously, one argues that $g$ is closed.

Now let us consider a composition $h_n\circ h_{n-1}\circ\ldots\circ h_2\circ h_1$, where each $h_i$ is equal to either $f$ or $g$.
Because $f[X]=g[X]=\textrm{EVEN}\cup\{a,b\}$ we have $h_1[X]=\textrm{EVEN}\cup\{a,b\}$. Because 
$$f[\textrm{EVEN}\cup\{a,b\}]=g[\textrm{EVEN}\cup\{a,b\}]=\{a,b\}$$
we have 
$$h_2\circ h_1[X]=\{a,b\}.$$
Because $f[\{a,b\}]=\{a\}$ and $g[\{a,b\}]=\{b\}$ we have for $n\geq3$
$$h_n\circ h_{n-1}\circ\ldots\circ h_2\circ h_1[X]=\{a\}$$
or
$$h_n\circ h_{n-1}\circ\ldots\circ h_2\circ h_1[X]=\{b\},$$
hence the IFS $\mathcal{F}=\{f,g\}$ is contractive.

Let $F$ be the Hutchinson operator induced by $\mathcal{F}$, namely $F:2^X\to2^X$ is defined as 
$$F(E):=f[E]\cup g[E].$$
We shall show that $F$ is not closed as a mapping from $2^X$ to $2^X$.

The one point set $\{a\}$ is not in the image $F[2^X]$ of the hyperspace $2^X$ via $F$. 
Thus it is enough to show that $\{a\}\in\overline{F[2^X]}$. Let $S(V_0;V_1,\ldots,V_k)$ be a base neigbourhood of $\{a\}$ which simply means here that $a\in\bigcap_{i=0}^kV_i$. Thus the set $\bigcap_{i=0}^kV_i$ is an open neighbourhood of $a$. Hence it must be a co-finite subset of $X$ and $2n\in\bigcap_{i=0}^kV_i$, for some $n\in\mathbb{N}$. This implies that $\{2n\}\in S(V_0;V_1,\ldots,V_k)$. We also have $\{2n\}=F(\{2n-1\})$. Hence $S(V_0;V_1,\ldots,V_k)\cap F[2^X]\neq\emptyset$. We conclude that $\{a\}\in \overline{F[2^X]}\setminus F[2^X]$.  
\end{proof}

\end{document}